\documentclass[11pt]{amsart}
\usepackage{hyperref}

\usepackage{enumerate}

\usepackage[utf8]{inputenc}
\usepackage[T1]{fontenc}

\usepackage[mathscr]{eucal}
\usepackage{amsmath,amsfonts,amsthm,amssymb}
\usepackage[all]{xy}

\usepackage[a4paper, hmargin={3cm,3cm}, centering]{geometry}

\everymath{\displaystyle}

\newtheorem{lemma}{Lemma}[section]
\newtheorem{theorem}[lemma]{Theorem}
\newtheorem*{theorem*}{Theorem}
\newtheorem{corollary}[lemma]{Corollary}

\newtheorem{proposition}[lemma]{Proposition}
\newtheorem*{proposition*}{Proposition}

\theoremstyle{remark}

\theoremstyle{definition}
\newtheorem*{definition*}{Definition}
\newtheorem*{conjecture*}{Conjecture}
\newtheorem*{remark*}{Remark}
\newtheorem*{remarks*}{Remarks}
\newtheorem*{lemma*}{Lemma}
\newtheorem*{claim*}{Claim}

\newcommand{\C}{{\mathbb C}}
\newcommand{\E}{{\mathbb E}}

\newcommand{\N}{{\mathbb N}}
\renewcommand{\P}{{\mathbb P}}
\newcommand{\Q}{{\mathbb Q}}
\newcommand{\R}{{\mathbb R}}
\newcommand{\T}{{\mathbb T}}
\newcommand{\Z}{{\mathbb Z}}

\newcommand{\CX}{{\mathcal X}}

\newcommand{\e}{\mathrm{e}}

\newcommand{\norm}[1]{\left\Vert #1\right\Vert}
\newcommand{\nnorm}[1]{\lvert\!|\!| #1|\!|\!\rvert}

\newcommand{\inv}{^{-1}}

\begin{document}

\title[Integer part independent polynomial averages and applications along primes]{Integer part independent polynomial averages and applications along primes}

\author{Dimitris Karageorgos and Andreas Koutsogiannis}
\address[Dimitris Karageorgos]{National and Kapodistrian University of Athens, Department of mathematics, Panepistemioupolis, Athens, Greece} \email{dimitris\_kar@math.uoa.gr}
\address[Andreas Koutsogiannis]{The Ohio State University, Department of mathematics, Columbus, Ohio, USA} \email{koutsogiannis.1@osu.edu}

\begin{abstract}
Exploiting the equidistribution properties of polynomial sequences,
 following the methods developed by Leibman (\cite{L}) and Frantzikinakis (\cite{F1} and \cite{F2}) we show that the ergodic averages with iterates given by the integer part of real-valued strongly independent polynomials, converge in the mean to the "right"-expected limit. These results have, via Furstenberg's correspondence principle, immediate combinatorial applications while combining these results with methods from \cite{FHK} and \cite{K2} we get the respective "right" limits and combinatorial results for multiple averages for a single sequence as well as for several sequences  along prime numbers.  
\end{abstract}



\maketitle

\section{Introduction}

The study of the limiting behaviour in $ L^2(\mu) $ of multiple ergodic averages 
of the form
 \begin{equation}\label{E:e4}
\frac{1}{N}\sum_{n=1}^N T^{a_1(n)}f_1\cdot \ldots\cdot T^{a_\ell(n)}f_\ell,
\end{equation}
where $ (a_1(n))_n,\ldots, (a_\ell(n))_n $ are sequences of integers, $T:X\to X$ is an invertible measure preserving transformation on the probability space $(X,\mathcal{B},\mu)$\footnote{We shall call the quadruple $(X,\mathcal{B},\mu,T)$ \emph{system}.} and $f_1,\ldots,f_\ell\in L^\infty(\mu), $ has been of great importance in the area of ergodic theory. The originator of this was Furstenberg (in \cite{Fu}) who first studied averages as in \eqref{E:e4} in the case where $ a_i(n)=in,$ $1\leq i\leq \ell,$ in order to provide an ergodic theoretical proof of Szemer\'{e}di's Theorem. 

Bergelson and Liebman in \cite{BL} studied the case where $ a_i$ are integer polynomials with no constant term and proved a polynomial extension of Szemer\'{e}di's theorem.

Another question that arises studying the limiting behaviour of \eqref{E:e4} is if the limit exists what can we say about it. In this direction and for the case of weakly mixing systems Fustenberg, Katznelson and Ornstein proved in \cite{FKO} that if $ a_i(n)=in,$ $1\leq i\leq\ell,$ the multiple ergodic averages in \eqref{E:e4} converge to the product of integrals of $ f_i $. Bergelson in \cite{B} extended this result in the case where $ a_i$ are essential distinct integer polynomials (i.e., all polynomials and their differences are not constant).  Furthermore, Frantzikinakis and Kra in \cite{FK} established the same result under the total ergodicity assumption of the system for  $ a_i $ independent integer polynomials (i.e., every non-trivial combination of $a_i$'s on the integers is not constant).

The convergence of multiple ergodic averages with several commuting transformations has been studied as well. Under weaker assumptions to the weakly mixing one on the system, the convergence to the "right"-expected limit, for the case of iterates of integer polynomials with different degrees and of integer part of real polynomials with different degrees, is treated in \cite{CFH} and \cite{K1} respectively, where by \emph{right limit} we mean that under the ergodic assumption, the limit is equal to the product of integrals of $f_i.$

We also have results in other, non-polynomial classes of iterates. Bergelson and H\r{a}land-Knutson in \cite{BK} treated the case of iterates of integer part Tempered functions on a weakly mixing system, while Frantzikinakis treated the case of iterates of integer part logarithmico-exponential Hardy field functions with polynomial degree of different growth (in \cite{F1} for a single $T$ and \cite{F4} for multiple commuting $T_i$'s while he also showed that in the sublinear, $0$-degree case, the commutativity assumption of $T_i$'s can be lifted), all showing convergence to the "right"-expected limit. We highlight at this point that the last results of Frantzikinakis are very strong and hold for general systems, a behavior that we didn't have for any class of polynomial iterates, with degree greater than $1,$ so far. 
 
In this paper we study the multiple convergence of ergodic averages with integer part of real polynomial iterates of  several sequences of the form
 \begin{equation}\label{E:fmc2}
\frac{1}{N}\sum_{n=1}^N T^{[p_1(n)]}f_1\cdot \ldots\cdot T^{[p_\ell(n)]}f_\ell.
\end{equation} 
We show in our main theorem, Theorem~\ref{T:ss}, that for the polynomial sequences $(p_1(n))_n,$ $\ldots,(p_\ell(n))_n,$ where every non-trivial linear combination of $p_i$'s has at least one non-constant irrational coefficient, under the ergodic assumption of $T$ (i.e., there is no non-trivial set invariant under $T$), the limit of \eqref{E:fmc2} as $N\to\infty$ in $L^2(\mu)$ is the "right"-expected one, meaning equal to $\prod_{i=1}^\ell \int f_i\;d\mu.$ So, using the ergodic decomposition, we have for a general system that the aforementioned limit is equal to the product of the conditional expectations of $f_i$'s, i.e., equal to $\prod_{i=1}^\ell \mathbb{E}(f_i\vert \mathcal{I}(T)).$ Measure theoretical and hence, via Furstenberg's correspondence principle, combinatorial applications of Theorem~\ref{T:ss} are given in Theorems~\ref{T:r1}, \ref{T:r2} and \ref{T:s11}. The strong nature of our results is also reflected in  Theorem~\ref{T:topol}, Corollary~\ref{C:topol}, Theorem\ref{T:comb} and Corollary~\ref{C:comb}, where we obtain some additional applications in topological dynamics and combinatorics respectively.

In order to prove Theorem~\ref{T:ss} we follow Frantzikinakis' approach (from \cite{F1} and \cite{F2}) and we also use some results of Leibman (\cite{L}) and Host-Kra (\cite{HK99}). More specifically, via Lemma~\ref{L:4.7} (Lemma~4.7 in \cite{F1}), which informs us that the nilfactor of our system is also the characteristic factor for the family of polynomials of interest, and the Structure theorem, Theorem~\ref{T:HK}, of Host and Kra  (\cite{HK99}), it suffices to show Theorem~\ref{T:ss} when our system is a nilsystem. To complete the proof, we use Theorem~\ref{T:T1}, an equidistribution result first proved by Frantzikinakis in \cite{F2} for logarithmico-exponential Hardy field functions with polynomial degree of different growth. In order to derive Theorem~\ref{T:T1}, we use Theorem~\ref{T:t1}, a central equidistribution result due to Leibman (\cite{L}) for a polynomial sequence in a connected and simply connected group.

Lastly, combining  Theorem~\ref{T:mc}, a result from \cite{F1} on multiple convergence of a single real polynomial sequence, and Theorem~\ref{T:ss} with some results from \cite{K2}, we get the analogous results in Theorems~\ref{T:ssp1} and \ref{T:ssp} respectively, together with their implications, for averages along prime numbers.

\medskip

Throughout the article we will highlight at some points the fact that one cannot expect to obtain the nice convergence and recurrent results that we get for other classical families of polynomials, say, integer polynomials, not even for very special families of them, fact that "forces" one to deal with real-valued polynomial families which satisfy some assumptions (see next section).

\subsection*{Notation} With $\N=\{1,2,\ldots\},$ $\mathbb{Z},$ $\mathbb{Q},$ $\mathbb{R}$ and $\mathbb{C}$ we denote the set of natural, integer, rational, real and complex numbers respectively. For $N\in\N$ we write $[1,N]$ to denote the set $\{1,\ldots,N\}.$ For a measurable function $f$ on a measure space $X$
with a transformation $T:X\to X,$ we denote with $Tf$ the composition $f\circ T.$  $\T^s =\R^s/\Z^s $ denotes the $s$ dimensional torus, $ e(t)=e^{2\pi it}$ denotes the exponential map, $(a(n))_n$ denotes a sequence indexed over the natural numbers (i.e., $(a(n))_{n\in \mathbb{N}}$), and  $ [\cdot] $ denotes the integer part (floor) function.

\section{Main results}

\begin{definition*}
Let $\ell\in\mathbb{N}$ and $a_1(t),\ldots,a_\ell(t)$ be real valued functions. We say that the family of sequences of integers $\{([a_1(n)])_n,\ldots,([a_\ell(n)])_n\}$ is \emph{good for multiple convergence} if for every system $(X,\mathcal{B},\mu,T)$ and functions $f_1,\ldots,f_\ell\in L^\infty(\mu)$ the limit \begin{equation}\label{E:fmc}
\lim_{N\to\infty}\frac{1}{N}\sum_{n=1}^N T^{[a_1(n)]}f_1\cdot \ldots\cdot T^{[a_\ell(n)]}f_\ell
\end{equation}
exists in $L^2(\mu).$
\end{definition*}

\begin{remark*}
It follows from Theorem~1.3 in \cite{K1} that any family of sequences coming from the integer part of polynomials with real
coefficients is good for multiple convergence.
\end{remark*}

\begin{definition*}\label{D:d1}
For $\ell\in\mathbb{N},$ let $\{p_1,\ldots,p_\ell\}$ be a family of real polynomials. We say that this family is \emph{strongly independent} (or that the polynomials $p_1,\ldots,p_\ell$ are \emph{strongly independent}) if any non-trivial linear combination of the polynomials $p_i$ has a non-constant irrational coefficient.

Note that a family with one element, $\{p\},$ where $p\in\mathbb{R}[t],$ is strongly independent iff  $p(t)\neq cq(t)+d$ for all $c,d\in\mathbb{R}$ and $q\in \Q[t]$ (or $\Z[t]$ equivalently).
\end{definition*}

\subsection*{Example.} The family of polynomials $\{\sqrt{2}t^2+t,\sqrt{3}t^2-t\}$ is strongly independent while the family $\{\sqrt{5}t^3+t^2+\sqrt{6}t, t^2, \sqrt{7}t\}$ is not.

\medskip

We also remark that our definition, in the case where the polynomial sequences are constant, coincide with the definition of the \emph{good} family of polynomials given in \cite{F3}, Problem 10.

\medskip

 Via our method of proof, it will become later clear that the assumptions on the real polynomials that we have are in a sense "optimal", since they are exactly what one has to assume in order to obtain the Weyl-type equidistribution results we have to show in order to prove our main result, i.e., that the limit of the ergodic averages over real strongly independent polynomials exists and it is the "right" one.

\begin{theorem}\label{T:ss}
Let $\ell\in\mathbb{N},$ $p_1,\ldots,p_\ell\in \mathbb{R}[t]$ be real strongly independent polynomials, $(X,\mathcal{B},\mu,T)$ be an ergodic system and $f_1,\ldots,f_\ell\in L^\infty(\mu).$ Then \begin{equation}\label{E:ss1}
\lim_{N\to\infty}\frac{1}{N}\sum_{n=1}^N T^{[p_1(n)]}f_1\cdot\ldots\cdot T^{[p_\ell(n)]}f_\ell=\prod_{i=1}^\ell \int f_i\;d\mu,
\end{equation}
where the convergence takes place in $L^2(\mu).$
\end{theorem}

\begin{remark*}
The assumption that the polynomials are strongly independent is necessary, since even for $\ell=1,$ $p(t)=\sqrt{2}t$ and ergodic rotations on the torus, \eqref{E:ss1} typically fails.
\end{remark*}

Note that even for a family of independent, integer polynomials it is not true in general that one has convergence as in \eqref{E:ss1}, i.e., to the right limit for a general ergodic system (see remark after Theorem~\ref{T:r1}), but has to have more assumptions on the system, as total ergodicity (see \cite{FK}). Hence, someone is "forced" to work with real polynomials in order to have this nice convergence behavior.

\medskip

We now state a principle due to Furstenberg, which allows one to obtain combinatorial results from ergodic theoretical ones. 


\begin{theorem*}[Furstenberg correspondence principle,~\cite{Fu}]
Let $E$ be a subset of integers. There exists a system $(X,\mathcal{B},\mu,T)$ and a set $A\in \mathcal{B},$ with $\mu(A)=\bar{d}(E)$ \footnote{For a set $ A \subseteq \N $ we define its \emph{upper density}, 
$ \bar{d}(A), $ as $ \bar{d}(A)=\limsup_{N\to \infty}\frac{
|A\cap  \{1,\ldots, N \}|}{N}. $ If the limit of the previous expression exists as $N\to\infty,$ we say that its value, denoted with $d(A),$ is the \emph{density} of $A.$} such that \begin{equation*}\label{E:FCP}
\bar{d}(E\cap(E-n_1)\cap\ldots\cap (E-n_\ell))\geq \mu(A\cap T^{-n_1}A\cap\ldots\cap T^{-n_\ell}A)
\end{equation*}
for every $n_1,\ldots,n_\ell\in \mathbb{Z}$ and $\ell\in \mathbb{N}.$
\end{theorem*}

As a consequence of Theorem~\ref{T:ss} we get the following recurrence result (for a proof of this result see for example Theorem~2.8 in \cite{F1}).

\begin{theorem}\label{T:r1}
Let $\ell\in\mathbb{N}$ and $p_1,\ldots,p_\ell\in \mathbb{R}[t]$ be real strongly independent polynomials. Then for every system $(X,\mathcal{B},\mu,T)$ and $A\in \mathcal{B}$ we have \begin{equation}\label{E:r1}
\lim_{N\to\infty}\frac{1}{N}\sum_{n=1}^N \mu\Big(A\cap T^{-[p_1(n)]}A\cap\ldots\cap T^{-[p_\ell(n)]}A\Big)\geq (\mu(A))^{\ell+1}.
\end{equation}
\end{theorem}

\begin{remark*}
The assumption that the polynomials are strongly independent is necessary since even for $\ell=1$ and $p(t)=t^2,$ \eqref{E:r1} typically fails.
\end{remark*}

The previous remark shows that \eqref{E:r1} typically fails even for families of independent, integer polynomials. Hence, Theorem~\ref{T:r1} is another indication that one has to work with real polynomials in order to have nice lower bounds as in \eqref{E:r1} for general systems. 

\medskip

Note at this point that our arguments show that the uniform version of Theorem~\ref{T:ss}, and hence its implications, holds, meaning that one can replace the standard Ces{\`a}ro averages, $\lim_{N\to\infty}\frac{1}{N}\sum_{n=1}^N,$ with the respective uniform ones, $\lim_{N-M\to\infty}\frac{1}{N}\sum_{n=M+1}^{N},$ and the natural upper density, $\bar{d},$  with the respective upper Banach density, $d^\ast$ \footnote{For a set $A\subseteq \Z,$ we define its \emph{upper Banach density}, $d^\ast(A),$ as $ d^\ast(A)=\limsup_{N-M\to \infty}\frac{
|A\cap  \{M+1,\ldots, N \}|}{N-M}. $}.

Then, one has that the uniform version of Theorem~\ref{T:r1} implies that for any $A\in \mathcal{B}$ with $\mu(A)>0,$ and every $\varepsilon>0$ the set $$R_\varepsilon(A)=\Big\{n\in \mathbb{Z}: \;\mu\Big(A\cap T^{-[p_1(n)]}A\cap\ldots\cap T^{-[p_\ell(n)]}A\Big)> (\mu(A))^{\ell+1}-\varepsilon\Big\}$$ is syndetic (i.e., it has bounded gaps).

We note that this general result, which holds under no assumption on the system, implies that a family of real strongly independent polynomials has a way different behavior than a family of linear integer polynomials, since for $p_i(t)=it$ we have that the syndeticity conclusion of the respective $R_\varepsilon(A)$ fails for certain ergodic systems when $\ell\geq 4,$ while for certain non-ergodic systems it fails even when $\ell\geq 2$ (for examples covering both cases, see \cite{BHK}).

\medskip

Using Theorem~\ref{T:r1} and Furstenberg's corresponding principle, we have the following.

\begin{theorem}\label{T:r2}
Let $\ell\in\mathbb{N}$ and $p_1,\ldots,p_\ell\in \mathbb{R}[t]$ be real strongly independent polynomials. Then for every $E\subseteq \mathbb{N}$ we have $$\liminf_{N\to\infty}\frac{1}{N}\sum_{n=1}^N \bar{d}(E\cap (E-[p_1(n)])\cap\ldots\cap (E-[p_\ell(n)]))\geq (\bar{d}(E))^{\ell+1}.$$
\end{theorem}


Immediate implication of the aforementioned result is the following.

\begin{theorem}\label{T:s11}
Let $\ell\in\mathbb{N}$ and $p_1,\ldots,p_\ell\in \mathbb{R}[t]$ be real strongly independent polynomials. Then every $E\subseteq \mathbb{N}$ with $\bar{d}(E)>0$ contains arithmetic configurations of the form
\begin{equation*}\label{E:s11}
\{m, m+[p_1(n)],m+[p_2(n)],\ldots,m+[p_\ell(n)]\}
\end{equation*}
for some $m\in \mathbb{Z}$ and $n\in\mathbb{N}$ with $[p_i(n)]\neq 0,$ for all $1\leq i\leq \ell.$
\end{theorem}

We note that someone can get the aforementioned result for integer polynomials with no constant term from the polynomial Szemer\'{e}di theorem (\cite{BL}), but in the generality that we present it here it is not clear to us at all if it follows from other already known results.

\medskip


In the next two applications of Theorem~\ref{T:ss} we follow the Subsections 2.3 and 2.4 from \cite{F4} respectively, where we get similar results for sequences of strongly independent polynomials instead of sequences of Hardy functions.

\subsubsection{An application in topological dynamics} Let $(X,T)$ be a (topological) dynamical system, where $(X,d)$ is a compact metric space and $T:X\to X$ an invertible continuous transformation. There exists a left-invariant, by $T,$ Borel measure which gives, in case $T$ is minimal (i.e., $\overline{\{T^n x:\;n\in\mathbb{N}\}}=X$ for all $x\in X,$ hence,  for every $x\in X$ and non-empty open set $U$ the set $\{n\in \mathbb{N}:\;T^nx\in U\}$ is syndetic)  positive
value to every non-empty open set. So, for almost every $x\in X$
and every non-empty open set $U$ we have 
\begin{equation}\label{E:ss1''}
\lim_{N\to\infty}\frac{1}{N}\sum_{n=1}^N {\bf{1}}_U(T^n x)>0.
\end{equation}

Note that from Theorem~\ref{T:ss}, using the ergodic decomposition, it follows that \begin{equation}\label{E:ss1'}
\lim_{N\to\infty}\frac{1}{N}\sum_{n=1}^N T^{[p_1(n)]}f_1\cdot\ldots\cdot T^{[p_\ell(n)]}f_\ell=\prod_{i=1}^\ell \mathbb{E}(f_i\vert\mathcal{I}(T)),
\end{equation}
where the convergence takes place in $L^2(\mu),$ $p_1,\ldots,p_\ell$ are real strongly independent polynomials, $f_1,\ldots,f_\ell\in L^\infty(\mu),$ $\mathcal{I}(T)$ denotes the $\sigma$-algebra of $T$-invariant sets and $\mathbb{E}(f\vert\mathcal{I}(T))$ is the conditional expectation with respect to $\mathcal{I}(T).$

Indeed, if $\mu=\int \mu_t\;d\lambda(t)$ denotes the ergodic decomposition of $\mu,$ it suffices to show that if $\mathbb{E}(f_i\vert \mathcal{I}(T))=0$ for some $i$ then the averages converge to $0.$ Since $\mathbb{E}(f_i\vert \mathcal{I}(T))=0,$ we have that $\int f_i \; d\mu_t=0$ for $\lambda$-a.e. $t.$ By \eqref{E:ss1}, we have that the averages go to $0$ in $L^2(\mu_t)$ for $\lambda$-a.e. $t,$ hence the limit is equal to $0$ in $L^2(\mu).$

 Since $ \mathbb{E}(f_i\vert\mathcal{I}(T))=\lim_{N\to\infty}\frac{1}{N}\sum_{n=1}^N T^n f_i,$ combining \eqref{E:ss1'} with \eqref{E:ss1''},  we
get for almost every $x\in X$ (and hence for a dense set) and every $U_1,\ldots,U_\ell$ from a given countable basis of non-empty open sets that $$\limsup_{N\to\infty}\frac{1}{N}\sum_{n=1}^N {\bf{1}}_{U_1}(T^{[p_1(n)]} x)\cdot\ldots\cdot {\bf{1}}_{U_\ell}(T^{[p_\ell(n)]} x)>0.$$ Using this we get:

\begin{theorem}\label{T:topol}
Let $\ell\in\mathbb{N},$ $p_1,\ldots,p_\ell\in\mathbb{R}[t]$ be real strongly independent polynomials and $(X,T)$ a minimal dynamical system. Then for a residual and $T$-invariant set of $x\in X$ we have 
\begin{equation}\label{E:topol}
\overline{\Big\{\Big(T^{[p_1(n)]}x,\ldots,T^{[p_\ell(n)]}x\Big):\;n\in\mathbb{N}\Big\}}=X\times\cdots\times X.
\end{equation}
\end{theorem}

\begin{remark*}
Even for $\ell=1$ examples of minimal
rotations on finite cyclic groups show that if $p\in\mathbb{Z}[t]$ is any polynomial different than $\pm t+d,$ then \eqref{E:topol} may fail for every $x\in X.$
\end{remark*}

Using Zorn's lemma we can easily show that every dynamical system has a minimal subsystem. So, using this and Theorem~\ref{T:topol} we get:

\begin{corollary}\label{C:topol}
Let $\ell\in\mathbb{N},$ $p_1,\ldots,p_\ell\in\mathbb{R}[t]$ be real strongly independent polynomials and $(X,T)$ a dynamical system. Then for a non-empty and $T$-invariant set of $x\in X$ we have 
\begin{equation}\label{E:topol'}
\overline{\Big\{\Big(T^{[p_1(n)]}x,\ldots,T^{[p_\ell(n)]}x\Big):\;n\in\mathbb{N}\Big\}}=\overline{\{T^n x:\;n\in\mathbb{N}\}}\times\cdots\times \overline{\{T^n x:\;n\in\mathbb{N}\}}.
\end{equation}
\end{corollary}

\begin{remark*}
As in the previous remark, examples for $\ell=1$ and $p\in \mathbb{Z}[t]$ with $p(t)\neq \pm t+d$ show that \eqref{E:topol'} typically fails. 
\end{remark*}

\subsubsection{An application in combinatorics} Using Theorem~\ref{T:ss}, we have the following recurrence result (for a proof use a similar argument as in Theorem~2.4 in \cite{F4}):

\begin{theorem}\label{T:rec}
Let $\ell\in\mathbb{N},$ $p_1,\ldots,p_\ell\in\mathbb{R}[t]$ be real strongly independent polynomials, $(X,\mathcal{B},\mu,T)$ a system and $A_0,A_1,\ldots,A_\ell\in \mathcal{B}$ such that $$\mu(A_0\cap T^{k_1}A_1\cap\ldots\cap T^{k_\ell}A_\ell)=\alpha>0$$ for some $k_1,\ldots,k_\ell\in\mathbb{Z}.$ Then $$\lim_{N\to\infty}\frac{1}{N}\sum_{n=1}^N \mu\Big(A_0\cap T^{-[p_1(n)]}A_1\cap\ldots\cap T^{-[p_\ell(n)]}A_\ell\Big)\geq \alpha^{\ell+1}.$$
\end{theorem}

 Using this result and a variant of Furstenberg's correspondence principle for several sets $A_i$ (see Proposition~3.3 from \cite{F5}) we get (see the $d=1$ case of Theorem~2.8 from \cite{F4}):
 
\begin{theorem}\label{T:comb}
Let $\ell\in\mathbb{N},$ $p_1,\ldots,p_\ell\in\mathbb{R}[t]$ be real strongly independent polynomials and $E_0,E_1,\ldots,E_\ell\subseteq \mathbb{N}$ that satisfy $$\overline{d}(E_0\cap (E_1+k_1)\cap\ldots\cap (E_\ell+k_\ell))=\alpha>0$$ for some $k_1,\ldots,k_\ell\in\mathbb{Z}.$ Then $$\liminf_{N\to\infty}\frac{1}{N}\sum_{n=1}^N \overline{d}(E_0\cap (E_1-[p_1(n)])\cap\ldots\cap (E_\ell-[p_\ell(n)]))\geq \alpha^{\ell+1}.$$
\end{theorem} 

We will sketch the proof of this result. In order to do so we recall a definition from \cite{F5}:

\begin{definition*}[\cite{F5}]
We say that the sequences $a_1,\ldots,a_\ell\in \ell^\infty(\mathbb{Z})$ \emph{admit correlations along
the sequence of intervals} $([1,N_k])_{k}$ with $N_k\to\infty$ as $k\to\infty,$ if the limit $$\lim_{k\to\infty} \frac{1}{N_k}\sum_{n=1}^{N_k}b_1(n+m_1)\cdot\ldots\cdot b_s(n+m_s)$$ exists for every $s\in\mathbb{N},$ $m_1,\ldots,m_s\in\mathbb{Z}$ and all sequences $b_1,\ldots,b_s\in \{a_1,\ldots,a_\ell, \overline{a}_1,\ldots,\overline{a}_\ell\}.$
\end{definition*}

We remark that for $a_1,\ldots,a_\ell\in \ell^\infty(\mathbb{Z}),$ using a diagonal argument, for any sequence $(N_k)_k\subseteq \mathbb{N}$ with $N_k\to\infty$ as $k\to\infty,$ we can find a subsequence $(N'_k)_k$ such that  $a_1,\ldots,a_\ell$ admit correlations along
the sequence of intervals $([1,N'_k])_{k}.$

\begin{proof}[Proof of Theorem~\ref{T:comb}]
Find a sequence of intervals ${\bf{N}}:=([1,N_k])_k$ along which the upper density of the intersection of the assumption is attained. Let $d_{{\bf{N}}}$ denote the corresponding density.  Passing to a subsequence, if needed, which we denote again by $([1,N_k])_k,$ we can assume that the functions ${\bf{1}}_{E_0},\ldots,{\bf{1}}_{E_\ell}$ admit correlations along
the sequence $([1,N_k])_k.$ Using Proposition~3.3 from \cite{F5}, we have that there exists a system $(X,\mathcal{B},\mu,T)$ and sets $A_0,\ldots, A_\ell\in \mathcal{B}$ such that  $$d_{{\bf{N}}}(E_0\cap(E_1-m_1)\cap\ldots\cap(E_\ell-m_\ell))=\mu(A_0\cap T^{m_1}A_1\cap\ldots\cap T^{m_\ell}A_\ell)$$ for all $m_1,\ldots,m_\ell\in \mathbb{Z}.$ The result now follows by Theorem~\ref{T:rec}.
\end{proof}

This result can be applied to several syndetic sets $E_0,E_1,\ldots,E_\ell\subseteq \mathbb{N}$ with constant $\alpha=\Big(\prod_{i=0}^\ell r_i \Big)^{-1},$ where $r_i$ is the syndeticity constant of $E_i,$ $0\leq i\leq\ell.$ So, one immediately gets the following:

\begin{corollary}\label{C:comb}
Let $\ell\in\mathbb{N},$ $p_1,\ldots,p_\ell\in\mathbb{R}[t]$ be real strongly independent polynomials and $E_0,E_1,\ldots,E_\ell\subseteq \mathbb{N}$ be syndetic sets. Then there exists $m,n\in\mathbb{N}$ such that $$m\in E_0,\; m+[p_1(n)]\in E_1,\;\ldots,\; m+[p_\ell(n)]\in E_\ell.$$
\end{corollary}

Via this last result, for a syndetic set $E\subseteq \mathbb{N},$ $p_1,\ldots,p_\ell\in\mathbb{R}[t]$ real strongly independent polynomials  and $c_0,c_1,\ldots,c_\ell\in \mathbb{N},$ setting $E_i=c_i E$\footnote{Where $cE:=\{cn:\;n\in E\}$.}, $0\leq i\leq \ell,$   we can find $x_0,x_1,\ldots,x_\ell\in E$ and $n\in \mathbb{N},$ solution of the following system of equations:
\begin{eqnarray*}
c_1x_1-c_0x_0 & = & [p_1(n)]\\
c_2x_2-c_0x_0 & = & [p_2(n)]\\
               & \vdots &\\
               c_\ell x_\ell -c_0x_0 & = & [p_\ell(n)].
\end{eqnarray*}

Let us note at this point that similar results fail even for $\ell=1,$ i.e., a single polynomial sequence and also fail when the set $E$ is only assumed to be piecewise
syndetic. Easy examples show that if $p\in \mathbb{Z}[t]$ is any polynomial different than $\pm t+d$ and $k\in \mathbb{N}\setminus\{1\},$ then the equation $kx-y=p(n)$ has no solution with $x, y$ belonging in some set $E$ that is an arithmetic
progression.

\subsection{Convergence along primes}

\medskip

Using Theorems~\ref{T:mc} (see below), ~\ref{T:ss} and some results from \cite{FHK2}, \cite{FHK} and \cite{K2}, we can have integer part polynomial multiple convergence along primes to the "right" limit for a strongly independent polynomial family. 

\subsubsection{Single sequence} The next result informs us that the limit of the ergodic averages with integer part of real polynomial iterates of a single sequence, is equal to the limit of the "Furstenberg averages" and it follows by Theorem~2.2 in \cite{F1}.

\begin{theorem}[\cite{F1}]\label{T:mc}
Let $p\in \mathbb{R}[t]$  with $p(t)\neq cq(t)+d,$ for all $c, d\in \mathbb{R}$ and $q\in \mathbb{Q}[t].$ Then for every $\ell\in\mathbb{N},$ system $(X,\mathcal{B},\mu,T)$ and $f_1,
\ldots,f_\ell\in L^\infty(\mu),$ we have \begin{equation}\label{E:mc3}
\lim_{N\to\infty}\frac{1}{N}\sum_{n=1}^N T^{[p(n)]}f_1\cdot T^{2[p(n)]}f_2\cdot\ldots\cdot T^{\ell[p(n)]}f_\ell=\lim_{N\to\infty}\frac{1}{N}\sum_{n=1}^N T^{n}f_1\cdot T^{2n}f_2\cdot\ldots\cdot T^{\ell n}f_\ell,
\end{equation} where the convergence takes place in $L^2(\mu).$
\end{theorem}

This theorem, via Furstenberg's correspondence principle, immediately implies the following Szemer{\'e}di type result:

\begin{theorem}[\cite{F1}]\label{T:s1}
Let $p\in\mathbb{R}[t]$ with $p(t)\neq cq(t)+d,$ where $c, d\in\mathbb{R}$ and $q\in \mathbb{Q}[t].$ Then for every $\ell\in\mathbb{N},$ every $E\subseteq \mathbb{N}$ with $\bar{\mathrm{d}}(E)>0$ contains arithmetic progressions of the form
\begin{equation*}\label{E:s1}
\{m, m+[p(n)],m+2[p(n)],\ldots,m+\ell[p(n)]\}
\end{equation*}
for some $m\in \mathbb{Z}$ and $n\in\mathbb{N}$ with $[p(n)]\neq 0.$
\end{theorem}

We will show the respective versions of these two last results along primes.

\begin{theorem}\label{T:ssp1}
Let $q\in \mathbb{R}[t]$  with $q(t)\neq c\tilde{q}(t)+d$ for all $c,d \in \mathbb{R}$ and $\tilde{q}\in \mathbb{Q}[t].$ Then for every  $\ell \in \N,$ system $(X,\mathcal{B},\mu,T)$ and $f_1,\ldots,f_\ell\in L^\infty(\mu),$ we have that 
\begin{equation*}
\lim_{N\to\infty}\frac{1}{\pi(N)}\sum_{p\in\P\cap [1,N]} T^{[q(p)]}f_1\cdot  T^{2[q(p)]}f_2\cdot\ldots\cdot T^{\ell[q(p)]}f_\ell=\lim_{N\to\infty}\frac{1}{N}\sum_{n=1}^N T^{n}f_1\cdot  T^{2n}f_2\cdot\ldots\cdot T^{\ell n}f_\ell,
  \end{equation*}
  where the convergence takes place in $L^2(\mu)$ and $\pi(N)=|\P\cap [1,N]|$ denotes the number of primes up to $N.$
\end{theorem} 


\begin{theorem}\label{T:s1p}
Let $q\in\mathbb{R}[t]$ with $q(t)\neq c\tilde{q}(t)+d,$ where $c, d\in\mathbb{R}$ and $\tilde{q}\in \mathbb{Q}[t].$ Then for every $\ell\in\mathbb{N},$ every $E\subseteq \mathbb{N}$ with $\bar{d}(E)>0$ contains arithmetic progressions of the form
\begin{equation*}\label{E:s1p}
\{m, m+[q(p)],m+2[q(p)],\ldots,m+\ell[q(p)]\}
\end{equation*}
for some $m\in \mathbb{Z}$ and $p\in\mathbb{P}$ with $[q(p)]\neq 0.$
\end{theorem}

\subsubsection{Several sequences} We also get the respective result and its applications of Theorem~\ref{T:ss} along primes:

\begin{theorem}\label{T:ssp}
Let $\ell \in \N,$ $p_1,\ldots,p_\ell\in \mathbb{R}[t]$ be real strongly independent polynomials, $(X,\mathcal{B},\mu,T)$ be an ergodic system and $f_1,\ldots,f_\ell\in L^\infty(\mu).$ Then \begin{equation*}\label{E:ssp}
\lim_{N\to\infty}\frac{1}{\pi(N)}\sum_{p\in\P\cap [1,N]} T^{[p_1(p)]}f_1\cdot\ldots\cdot T^{[p_\ell(p)]}f_\ell=\prod_{i=1}^\ell \int f_i\;d\mu,
\end{equation*}
  where the convergence takes place in $L^2(\mu).$ 
\end{theorem} 

Theorem~\ref{T:ssp} has the following implications, analogous to the ones that Theorem~\ref{T:ss} has.

\begin{theorem}\label{T:r1p}
Let $\ell\in\mathbb{N}$ and $p_1,\ldots,p_\ell\in \mathbb{R}[t]$ be real strongly independent polynomials. Then for every system $(X,\mathcal{B},\mu,T)$ and $A\in \mathcal{B}$ we have \begin{equation*}\label{E:r1p}
\lim_{N\to\infty}\frac{1}{\pi(N)}\sum_{p\in\P\cap[1,N]} \mu\Big(A\cap T^{-[p_1(p)]}A\cap\ldots\cap T^{-[p_\ell(p)]}A\Big)\geq (\mu(A))^{\ell+1}.
\end{equation*}
\end{theorem}

\begin{theorem}\label{T:r2p}
Let $\ell\in\mathbb{N}$ and $p_1,\ldots,p_\ell\in \mathbb{R}[t]$ be real strongly independent polynomials. Then for every $E\subseteq \mathbb{N}$ we have $$\liminf_{N\to\infty}\frac{1}{\pi(N)}\sum_{p\in\P\cap[1,N]} \bar{d}(E\cap (E-[p_1(p)])\cap\ldots\cap (E-[p_\ell(p)]))\geq (\bar{d}(E))^{\ell+1}.$$
\end{theorem}

\begin{theorem}\label{T:s11p}
Let $\ell\in\mathbb{N}$ and $p_1,\ldots,p_\ell\in \mathbb{R}[t]$ be real strongly independent polynomials. Then every $E\subseteq \mathbb{N}$ with $\bar{d}(E)>0$ contains arithmetic configurations of the form
\begin{equation*}\label{E:s11p}
\{m, m+[p_1(p)],m+[p_2(p)],\ldots,m+[p_\ell(p)]\}
\end{equation*}
for some $m\in \mathbb{Z}$ and $p\in\mathbb{P}$ with $[p_i(p)]\neq 0,$ for all $1\leq i\leq \ell.$
\end{theorem}

We close this section with the remark that we believe that the respective reformulations of the results stated in this section for a single transformation hold for several commuting transformations but the method we use does not allow us to prove them in this more general setting.  

\section{Background Material}

\subsection{Nilmanifolds}
In this subsection we recall some basic facts on nilmanifolds and equidistribution results on them.

\subsubsection{Definitions and basic properties}
Let $ G $ be a $ k $-step nilpotent Lie group, meaning $ G_{k+1}=\{e\} $ for some $ k\in \N $, where $ G_k=[G,G_{k-1}] $ denotes the $ k $-th commutator subgroup, and $ \Gamma $ is a discrete cocompact subgroup of $ G $. The  compact homogeneous space
$ X=G/\Gamma $ is called a $ k $-\emph{nilmanifold} (or just \emph{nilmanifold}).

The group $ G $ acts on $ G/\Gamma $ by left translation where the translation  by an element $ b\in G $ is given by $ T_b(g\Gamma)=(bg)\Gamma $. We denote by $ m_X $ the normalized \emph{Haar measure} on $ X $, meaning, the unique probability measure that is invariant under the action of $ G $ by left translations and $ \mathcal{G}/\Gamma $ denotes the Borel $ \sigma $-algebra of $ G/\Gamma $. If $ b\in G $, we call the system $ (G/\Gamma, \mathcal{G}/\Gamma,m_X,T_b) $ a $ k $-\emph{step nilsystem} (or just \emph{nilsystem}) and the elements of  $ G $ \emph{nilrotations}.

\subsubsection{Equidistribution on nilmanifolds} For a connected and simply connected Lie group $G,$ let $ \exp :\mathfrak{g}\to G $ be the exponential map, where $ \mathfrak{g} $ is the Lie algebra of $ G $. For $ b \in G $ and $ s\in \R $ we define the element $ b^s $ of $ G $ as follows: If $ X\in \mathfrak{g} $ is such that $ \exp (X)=b $, then $ b^s=\exp(sX) $ (this is well defined since under the aforementioned assumptions $ \exp$ is a bijection).

 If $ (a(n))_{n} $ is a sequence of real numbers and $ X=G/\Gamma $ is a nilmanifold with $ G $ connected and simply connected, we say that the sequence
$ (b^{a(n)}x)_{n} $ is \emph{equidistributed} in a sub-nilmanifold $ Y $ of $ X $, if for every $ F\in C(X) $ we have
\[ \lim_{N\to \infty} \frac{1}{N}\sum_{n=1}^N F(b^{a(n)}x)=\int F\; d m_Y. \]
If the sequence $ (a(n))_{n} $ takes only integer values, we are not obliged to assume that $ G $ is connected and simply connected.


A nilrotation $ b\in G $ is \emph{ergodic}, or \emph{acts ergodically} on $ X $, if the sequence $ (b^n\Gamma)_{n} $ is dense in $X.$ If $ b\in G $ is ergodic, then for every
$ x\in X $ the sequence $ (b^nx)_{n} $ is equidistributed in $ X $.



Let $ X=G/\Gamma $ be a nilmanifold and $ b\in G $. Then the orbit closure
$\overline{(b^n\Gamma)}_{n} $ of $ b $ has the structure of a nilmanifold. Furthermore, the sequence $ (b^n\Gamma)_{n} $ is equidistributed in $ \overline{(b^n\Gamma)}_{n}  $. If $ G $ is connected and simply connected and $ b\in G $, then
$ \overline{(b^s\Gamma)}_{s\in \R} $ is a nilmanifold. Furthermore, the nilflow $ (b^s\Gamma)_{s\in \R} $ is equidistributed in $ \overline{(b^s\Gamma)}_{s\in \R} $.



If $ G $ is a nilpotent group, then a sequence $ g:\N\to G $ of the form
$ g(n)=b_1^{p_1(n)}\cdots b_k^{p_k(n)}, $ where $ b_i\in G $ and $ p_i $ are polynomials taking integer values at the integers for every $1\leq i\leq k $ is called a \emph{polynomial sequence} in $ G $.
A \emph{polynomial sequence on the nilmanifold} $ X= G/\Gamma $ is a sequence of the form
$ (g(n)\Gamma)_{n} $ where $ g:\N\to G $ is a polynomial sequence in $ G $.

\medskip

The following qualitative equidistribution result was established by Leibman in \cite{L}:

\begin{theorem}[\textbf{Leibman}, \cite{L}]\label{T:t1}
Suppose that $ X=G/\Gamma $ is a nilmanifold with $ G $ connected and simply connected and
$ (g(n))_{n} $ is a polynomial sequence in $ G $. Let $ Z=G/([G,G]\Gamma) $ and
$ \pi :X\to Z $ be the natural projection. Then the following statements hold:
\begin{enumerate}
\item For every $ x\in X $ the sequence $ (g(n)x)_{n} $ is equidistributed in  a finite union of subnilmanifolds $ X $.
\item For every $ x\in X $ the sequence $ (g(n)x)_{n} $ is equidistributed in $ X $ if and only if the sequence $ (g(n)\pi(x))_{n} $ is equidistributed in $ Z $.
\end{enumerate}
\end{theorem}

Let $ X=G/\Gamma $ be a nilmanifold with $ G $ connected and simply connected, then $ Z $ is a connected compact abelian Lie group, hence a torus, meaning $ \T^s $ for some $ s \in \N $, and as a consequence every nilrotation in $ Z $ is isomorphic to a rotation on $ \T^s $.


\subsection{Ergodic Theory}
We gather below some basic notions and facts from ergodic theory that we use throughout the paper.

\subsubsection{Factors} A \emph{homomorphism} from a system $ (X,\CX, \mu, T) $ onto a system $ (Y, \mathcal{Y}, \nu, S) $ is a measurable map $ \pi :X\to Y $, such that $ \mu \circ \pi \inv =\nu $ and $ S\circ \pi(x)=\pi \circ T(x) $ for $ x\in X $. When we have such a homomorphism we say that the system $ (Y, \mathcal{Y}, \nu, S) $ is a \emph{factor} of the system $ (X,\CX, \mu, T) $. If the factor map $ \pi :X\to Y $ can be chosen to be injective, then we say that the systems $ (X,\CX, \mu, T) $ and $ (Y, \mathcal{Y}, \nu, S) $ are \emph{isomorphic}.
A factor can also be characterised by $ \pi \inv(\mathcal{Y}) $ which is a $ T $-invariant sub-$ \sigma $-algebra of $ \mathcal{X} $. By a classical abuse of terminology we denote by the same letter the $ \sigma $-algebra $ \mathcal{Y} $ and $ \pi \inv(\mathcal{Y}) $.

\subsubsection{Characteristic Factors}
Let $ (X,\CX, \mu, T) $ be a system. We say that the $ \sigma $-algebra $ \mathcal{Y} $ of $ \mathcal{X} $ is a \emph{characteristic factor} for the family of integer sequences
$ \{ (a_1(n))_n,\ldots, (a_\ell(n))_n \} $
if $ \mathcal{Y} $ is $ T $-invariant and
\[\lim_{N\to\infty} \norm{\frac{1}{N}\sum_{n=1}^N T^{a_1(n)}f_1\cdot\ldots\cdot T^{a_\ell(n)}f_\ell- 
\frac{1}{N}\sum_{n=1}^N T^{a_1(n)}\tilde{f}_1\cdot\ldots\cdot T^{a_\ell(n)}\tilde{f}_\ell }_{L^2(\mu)}= 0, \]
 where $ \tilde{f}_i=\E(f_i|\mathcal{Y}) $, for $ f_i\in L^\infty(\mu) $ for all
$1\leq i\leq \ell $ \footnote{Equivalently, if $ \E(f_i|\mathcal{Y})=0 $ for some $1\leq  i\leq  \ell  $, then $ \lim_{N\to\infty}\norm{\frac{1}{N}\sum_{n=1}^N T^{a_1(n)}f_1\cdot\ldots\cdot T^{a_\ell(n)}f_\ell}_{L^2(\mu)}= 0 $. }.

\subsubsection{Seminorms and Nilfactors}

We follow \cite{HK99} and \cite{CFH} for the inductive definition of the seminorms $\nnorm{\cdot}_k.$ More specifically, the definition that we use here follows from \cite{HK99} (in the ergodic case), \cite{CFH} (in the general case) and the use of von Neumann's ergodic theorem.

Let $(X,\mathcal{B},\mu,T)$ be a system and $f\in L^\infty(\mu).$  We define inductively the seminorms $\nnorm{f}_k$ as follows: For $ k=1 $ we set
$$ \nnorm{f}_1:= \norm{\E(f|\mathcal{I}(T))}_{L^2(\mu)}.$$
For $k\geq 1,$ we let
$$\nnorm{f}^{2^{k+1}}_{k+1}:=\lim_{N\to\infty}\frac{1}{N}\sum_{n=1}^{N}\nnorm{\bar{f}\cdot T^n f}^{2^k}_k . $$
It was shown in \cite{HK99} that for every integer $ k\geqslant 1 $ all these limits exist and $\nnorm{\cdot}_k $ defines a seminorm on $ L^\infty(\mu) $.

Using these seminorms we can construct factors $ \mathcal{Z}_k=\mathcal{Z}_k(T) $ of $ X $ characterized by  the property:
\[ \text{ for } f\in L^\infty(\mu),\;\;\; \E(f|\mathcal{Z}_{k-1})=0 \text{ if and only if } \nnorm{f}_k=0. \]

It was also shown in \cite{HK99} that for every $ k \in \N $ the factor
$ \mathcal{Z}_k $ has an algebraic structure, in fact we can assume that it is a $ k $-step nilsystem. This is the content of the following Structure theorem, which we recall in the ergodic case:

\begin{theorem}[\textbf{Host \& Kra}, \cite{HK99}]\label{T:HK}
Let $(X,\mathcal{B},\mu,T)$ be an ergodic system and $ k\in\mathbb{N}$. Then the factor $\mathcal{Z}_k(T)$ is an inverse limit of $ k $-step nilsystems \footnote{By this we mean that there exist $T$-invariant sub-$\sigma$-algebras $\mathcal{Z}_{k,i}, i\in \N$, of $\mathcal{B}$ such that $\mathcal{Z}_k=\bigcup_{i\in\N}\mathcal{Z}_{k,i}$ 
and for every $i\in \N$, the factors induced by  the $\sigma$-algebras $\mathcal{Z}_{k,i}$ are isomorphic to $k$-step nilsystems.}.
\end{theorem}

Because of this result we call $ \mathcal{Z}_k $ the $ k $-\emph{step nilfactor} of the system. The smallest factor that is an extension of all finite step nilfactors is denoted by $ \mathcal{Z}=\mathcal{Z}(T) $, meaning, $\mathcal{Z}=\bigvee_{k\in\mathbb{N}}\mathcal{Z}_k $, and is called the \emph{nilfactor} of the system.

\section{Equidistribution Results}

In this section we establish some equidistribution results in order to prove the convergence and recurrence results stated in Section 2. In order to obtain these equidistibution results we follow the main strategy introduced in \cite{F2} (Sections 4 and 5).

First, we give an equidistribution result involving nil-orbits of several sequences of strongly independent polynomials.

\begin{theorem}\label{T:T1}
Let $\ell\in\mathbb{N}$ and $p_1,\ldots,p_\ell\in\mathbb{R}[t]$ be real strongly independent polynomials.
\begin{enumerate}
\item If $X_i=G_i/\Gamma_i,$ $1\leq i\leq \ell,$ are nilmanifolds with $G_i$ connected and simply connected, then for every $b_i\in G_i$ and $x_i\in X_i$ the sequence $$\Big(b_1^{p_1(n)}x_1,\ldots, b_\ell^{p_\ell(n)}x_\ell\Big)_{n}$$ is equidistributed in the nilmanifold $$\overline{(b_1^sx_1)}_{s\in\mathbb{R}}\times\cdots\times \overline{(b_\ell^sx_\ell)}_{s\in\mathbb{R}}.$$
\item If $X_i=G_i/\Gamma_i,$ $1\leq i\leq \ell,$ are nilmanifolds, then for every $b_i\in G_i$ and $x_i\in X_i$ the sequence $$\Big(b_1^{[p_1(n)]}x_1,\ldots, b_\ell^{[p_\ell(n)]}x_\ell\Big)_{n}$$ is equidistributed in the nilmanifold $$\overline{(b_1^nx_1)}_{n}\times\cdots\times \overline{(b_\ell^nx_\ell)}_{n}.$$
\end{enumerate}
\end{theorem}

Notice that in order to prove Theorem~\ref{T:T1}, we can assume that $X_1=\ldots=X_\ell=X.$
Indeed, in the general case we consider the nilmanifold $\tilde{X}=X_1\times\cdots\times X_\ell.$ Then $\tilde{X}=\tilde{G}/\tilde{\Gamma},$ where $\tilde{G}=G_1\times \cdots\times G_\ell$ is connected and simply connected and $\tilde{\Gamma}=\Gamma_1\times\cdots\times\Gamma_\ell$ is a discrete cocompact subgroup of $\tilde{G}.$ Each $b_i$ can be considered as an element of $\tilde{G}$ and each $x_i$ as an element of $\tilde{X}.$

The following lemma shows that Part (ii) of Theorem~\ref{T:T1} follows from Part (i).

\begin{lemma}[Lemma~5.1, \cite{F2}]\label{L:5.1}
Let $\ell\in\mathbb{N}$ and $(a_1(n))_n,\ldots,(a_\ell(n))_n$ be sequences of real numbers. Suppose that for every nilmanifold $X=G/\Gamma,$ with $G$ connected and simply connected and every $b_1,\ldots,b_\ell\in G$ the sequence $$\Big(b_1^{a_1(n)}\Gamma,\ldots,b_\ell^{a_\ell(n)}\Gamma\Big)_{n}$$ is equidistributed in the nilmanifold  $$\overline{(b_1^s\Gamma)}_{s\in\mathbb{R}}\times\cdots\times \overline{(b_\ell^s\Gamma)}_{s\in\mathbb{R}}.$$ Then for every nilmanifold $X=G/\Gamma,$ every $b_1,\ldots,b_\ell\in G$ and $x_1,\ldots,x_\ell\in X$ the sequence $$\Big(b_1^{[a_1(n)]}x_1,\ldots, b_\ell^{[a_\ell(n)]}x_\ell\Big)_{n}$$ is equidistributed in the nilmanifold $$\overline{(b_1^nx_1)}_{n}\times\cdots\times \overline{(b_\ell^nx_\ell)}_{n}.$$
\end{lemma}


Next we give a result needed to prove Part (i) of Theorem~\ref{T:T1}.

\begin{proposition}\label{P:1}
Let $\ell\in\mathbb{N}$ and $p_1,\ldots,p_\ell\in\mathbb{R}[t]$  be real strongly independent polynomials. Let $X_i=G_i/\Gamma_i,$ $1\leq i\leq\ell,$ be nilmanifolds with $G_i$ connected and simply connected and elements $b_i\in G_i$ acting ergodically on $X_i.$ Then the sequence $$\Big(b_1^{p_1(n)}\Gamma_1,\ldots,b_\ell^{p_\ell(n)}\Gamma_\ell\Big)_{n}$$ is equidistributed in the nilmanifold $X_1\times\cdots\times X_\ell.$
\end{proposition}

\begin{proof}
For convenience, since the general case is similar, we assume that $X_1=\ldots=X_\ell=X.$ First notice that the sequence $\Big(b_1^{p_1(n)},\ldots,b_\ell^{p_\ell(n)}\Big)_{n}$ is a polynomial sequence in $G^\ell.$ Since $X^\ell=G^\ell/\Gamma^\ell$ with $G^\ell$ connected and simply connected, we can apply Theorem~\ref{T:t1}. So, in order to prove that  $\Big(b_1^{p_1(n)}\Gamma,\ldots,b_\ell^{p_\ell(n)}\Gamma\Big)_{n}$ is equidistributed in $G^\ell $,
it suffices to show that  $\Big(\pi(b_1^{p_1(n)}\Gamma),\ldots,\pi(b_\ell^{p_\ell(n)}\Gamma)\Big)_{n}$ is equidistributed in $Z^\ell,$ where $Z=G/([G,G]\Gamma)$ and $\pi:X\to Z$ is the natural projection.

Since $ G $ is connected and simply connected $ Z $ is isomorphic to some finite dimensional torus $ \T^s $ and as a consequence every nilrotation in $ Z $ is isomorphic to a rotation on $ \T^s $.
Hence, for every $1\leq i\leq \ell$ we have $\pi(b_i\Gamma)=(\beta_{i,1}\mathbb{Z},\ldots,\beta_{i,s}\mathbb{Z})$, where $\beta_{i,j}\in \mathbb{R}$ and $(\beta_{i,1},\ldots,\beta_{i,s})$ is the projection of $b_i$ on $\mathbb{T}^s$ (note that the $s$ is bounded by the dimension of $X$). Since every $b_i$ acts ergodically on $X$, we have that for all $1\leq i\leq \ell$ the set of real numbers $\{1,\beta_{i,1},\ldots,\beta_{i,s}\}$ is rationally independent. Also, for every $t\in \mathbb{R}$ and $1\leq i\leq \ell$ we have that $\pi(b^t_i\Gamma)=(t\beta'_{i,1}\mathbb{Z},\ldots,t\beta'_{i,s}\mathbb{Z})$, for some $\beta'_{i,j}\in \mathbb{R}$ with $\beta'_{i,j}\mathbb{Z}=\beta_{i,j}\mathbb{Z},$ and so
$\pi\Big(b^{p_i(n)}_i\Gamma\Big)=(p_i(n)\beta'_{i,1}\mathbb{Z},\ldots,p_i(n)\beta'_{i,s}\mathbb{Z})$. Note that $\{1,\beta'_{i,1},\ldots,\beta'_{i,s}\}$ is also a rationally independent set for all $1\leq i\leq \ell.$

Our objective now is to establish the equididstribution of the sequence
\[ \left((p_1(n)\beta'_{1,1}\Z, \ldots, p_1(n)\beta'_{1,s}\Z, \ldots, p_\ell(n)\beta'_{\ell,1}\Z, \ldots, p_\ell(n)\beta'_{\ell,s}\Z )\right)_{n} \] 
on $ \T^{\ell s} $. To verify this we use Weyl's criterion (\cite{W}).
 
Let $ {\bf{h}}=(h_{1,1},\ldots,h_{1,s},\ldots,h_{\ell,1},\ldots,h_{\ell,s})\in \mathbb{Z}^{\ell s}\setminus \{(0,\ldots, 0)\} $. Because of the aforementioned rational independence, not all the sums $\sum_{j=1}^s h_{i,j}\beta'_{i,j},$ $1\leq i\leq \ell,$ are equal to $0,$ so, using the strong independence of the family of polynomials $\{p_1,\ldots,p_\ell\}$ we get that the polynomial $\sum_{i=1}^\ell \Big(\sum_{j=1}^s h_{i,j}\beta'_{i,j}\Big)p_i(n)$ has at least one non-constant irrational coefficient. So
$$\lim_{N\to\infty}\frac{1}{N}\sum_{n=1}^N \e\left({\bf{h}}\cdot(p_{1}(n)\beta'_{1,1},\ldots,p_{1}(n)\beta'_{1,s},\ldots,p_{\ell}(n)\beta'_{\ell,1},\ldots,p_{\ell}(n)\beta'_{\ell,s})\right)= $$
$$\lim_{N\to\infty}\frac{1}{N}\sum_{n=1}^N e\Big(\sum_{i=1}^\ell \Big(\sum_{j=1}^s h_{i,j}\beta'_{i,j}\Big)p_i(n)\Big)=0.$$ 
By Weyl's equidistribution criterion, the result follows.
\end{proof}

The last ingredient in proving Part  (i)  of Theorem~\ref{T:T1} is the following lemma:

\begin{lemma}[Lemma~5.2, \cite{F2}]\label{L:5.2}
Let $X=G/\Gamma$ be a nilmanifold with $G$ connected and simply connected. Then for every $b_1,\ldots,b_\ell\in G$ there exists an $s_0\in \mathbb{R}$ such that for all $1\leq i\leq \ell$ the element $b_i^{s_0}$ acts ergodically on the nilmanifold $\overline{(b_i^s \Gamma)}_{s\in\mathbb{R}}.$
\end{lemma}

We are now ready to prove Theorem \ref{T:T1}.

\begin{proof}[Proof of Theorem~\ref{T:T1}]
Using Lemma \ref{L:5.1} we see that Part (ii) of Theorem \ref{T:T1} follows from Part
(i). To establish Part (i) let $ b_1,\ldots, b_\ell \in G $. By Lemma \ref{L:5.2} there exists a non-zero $ s_0\in \R $ such that for every $1\leq i\leq \ell$ the element $b_i^{s_0}$ acts ergodically on the nilmanifold $\overline{(b_i^s \Gamma)}_{s\in\mathbb{R}}.$ Using Proposition \ref{P:1} for the elements $ b_i^{s_0} $ and the polynomials $ p_i(s)/s_0 $ (who trivially are still strongly independent) we get that the sequence
$\Big(b_1^{p_1(n)}\Gamma,\ldots,b_\ell^{p_\ell(n)}\Gamma\Big)_{n}$ is equidistributed in the nilmanifold  $\overline{(b_1^s\Gamma)}_{s\in\mathbb{R}}\times\cdots\times \overline{(b_\ell^s\Gamma)}_{s\in\mathbb{R}}$, hence we get the conclusion.
\end{proof}






\section{Proof of main results}

In this last section we present the proofs of our main results, namely Theorems \ref{T:ss}, \ref{T:ssp1} and \ref{T:ssp}.

We first give the proof of Theorem \ref{T:ss}. In order to do so we recall from \cite{F1} that  the nilfactor $ \mathcal{Z} $ of a system is characteristic for the family $ \{ p_1,\ldots, p_\ell \} $, where $ p_1,\ldots, p_\ell\in \R[t] $ are real strongly independent polynomials.
 Actually, the statement in \cite{F1} is about \emph{nice} families of polynomials (see definition below), a notion more general than the strong independence that we use here. 

\begin{definition*}[\cite{F1}]
Let $\ell \in\mathbb{N}$ and for $N\in\mathbb{N},$ let $\mathcal{P}_N = \{p_{1,N},\ldots, p_{\ell,N}\}$ be a family of polynomials with real coefficients. We say that the collection $(P_N)_N$ is \emph{nice} if for every $N\in\mathbb{N}$ the polynomials $p_{i,N}$ and $p_{i,N}-p_{j,N},$ $i\neq j,$ 
are non-constant and their leading coefficients are independent of $N.$
\end{definition*}

The following lemma shows that for a nice collection of polynomial families the nilfactor is the characteristic factor as well  (a different proof of this fact is also given in \cite{L1}).

\begin{lemma}[Lemma~4.7, \cite{F1}]\label{L:4.7}
Let $(\{p_{1,N},\ldots,p_{\ell,N}\})_N$ be a nice collection of polynomial families, $(X,\mathcal{B},\mu,T)$ be a system and suppose that one of the functions $f_1,\ldots,f_\ell\in L^\infty(\mu)$ is orthogonal to the nilfactor $\mathcal{Z}.$ Then for any F{\o}lner sequence $(\Phi_N)_N$ in $\mathbb{Z}$ \footnote{A \emph{F{\o}lner sequence in $\mathbb{Z}$} is a sequence $(\Phi_n)_n$ of finite sunsets of $\Z$ that for any $m\in\Z$ we have $\lim_{n\to\infty}\frac{|(\Phi_n+m)\cap \Phi_n|}{|\Phi_n|}=1.$} and any bounded two parameter sequence $(c_{N,n})_{N,n}$ of real numbers we have \begin{equation}\label{E:np}
\lim_{N\to\infty} \frac{1}{|\Phi_N|}\sum_{n\in \Phi_N} c_{N,n} T^{[p_{1,N}(n)]}f_1\cdot\ldots\cdot T^{[p_{\ell,N}(n)]}f_\ell=0,
\end{equation}
where the convergence takes place in $L^2(\mu).$
\end{lemma}

\begin{remark*}
For $\ell\in\mathbb{N},$ let a strongly independent family of polynomials $\{p_1,\ldots,p_\ell\}.$ Then, trivially this collection is nice, so we have \eqref{E:np}, hence the nilfactor $\mathcal{Z}$ is the characteristic factor for this family. 
\end{remark*}

\begin{proof}[Proof of Theorem~\ref{T:ss}]
We start by using Lemma \ref{L:4.7} in order to get that the nilfactor
$ \mathcal{Z} $ is characteristic  for the corresponding multiple ergodic average. Via Theorem~\ref{T:HK} we can assume without loss of generality that our system is an inverse limit of nilsystems. By a standard approximation argument, we can further assume that our system is a nilsystem.

Let $(X=G/\Gamma,\mathcal{G}/\Gamma,m_X,T_b)$ be a nilsystem, where $ b\in G $ is ergodic, and $ F_1,\ldots,F_\ell\in L^\infty(m_X) $. Our objective now is show that if $ \{p_1,\ldots, p_\ell \} $ is a strongly independent family of polynomials then 
\begin{equation}\label{E:e3}
\lim_{N\to \infty}\sum_{n=1}^N F_1(b^{[p_1(n)]}x)\cdot \ldots \cdot F_\ell(b^{[p_\ell(n)]}x)= \int F_1 \; dm_X \cdot \ldots \cdot \int F_\ell \; dm_X
\end{equation}
where the convergence takes place in $ L^2(m_X) $. By density we can assume that the functions $ F_1,\ldots,F_\ell $ are continuous. Then we can apply Theorem~\ref{T:T1} to the nilmanifold $ X^\ell $, the nilrotation
$ \tilde{b}=(b,\ldots, b)\in G^\ell $, the point
$ \tilde{x}=(x,\ldots, x)\in X^\ell $, and the continuous function
$ \tilde{F}(x_1,\ldots,x_\ell)=F_1(x_1)\cdot\ldots\cdot F_\ell(x_\ell) $, we get that
\[ \lim_{N\to \infty}\sum_{n=1}^N \tilde{F}(b^{[p_1(n)]}x, \ldots, b^{[p_\ell(n)]}x)= \int \tilde{F} \; dm_{X^\ell} \]
and this gives the desired limit in \eqref{E:e3}, completing the proof.
\end{proof}


\subsection{Convergence along Primes}
We first give the definitions and the main ideas in order to prove the Theorems~\ref{T:ssp1} and ~\ref{T:ssp}. 

\medskip

We start by recalling the definition of the {\em von Mangoldt function}, $\Lambda:\N\to\R,$ where $ \Lambda(n)=\left\{ \begin{array}{ll} \log(p) \quad ,\text{ if } n=p^k \text{ for some } p\in \P \text{ and some }k\in \N\\ 0 \;\;\; \; \quad \quad ,\text{ elsewhere }\end{array} \right. .$ 

\medskip

As in \cite{FHK} and \cite{K2} it is more natural for us to deal, in stead of $\Lambda,$ with the function $\Lambda':\N\to\R,$ where $\Lambda'(n)={\bf 1}_{\P}(n)\cdot \Lambda(n)={\bf 1}_{\P}(n)\cdot \log(n).$

\medskip

The function $\Lambda',$ according to the following lemma, will allow us to relate averages along primes with weighted averages over the integers.

\begin{lemma}[\cite{FHK2}]\label{L:n}
If $a:\N\to\C$ is bounded, then
$$
\lim_{N\to\infty}\left| \frac{1}{\pi(N)}\sum_{p\in \P\cap[1,N]}a(p)-\frac{1}{N}\sum_{n=1}^N \Lambda'(n)\cdot a(n)\right|=0.$$
\end{lemma}

\noindent For $w>2,$ let 
$$
W=\prod_{p\in\P\cap[1,w-1]}p$$ be the product of primes bounded above by $w.$ For $r\in \N,$ let 
$$
\Lambda_{w,r}'(n)=\frac{\phi(W)}{W}\cdot \Lambda'(Wn+r),$$ where $\phi$ is the Euler function, be the {\em modified von Mangoldt function}.

\begin{definition*}
For $\ell \in \N,$ we call the setting $(X,\mathcal{B},\mu,T_1,\ldots,T_\ell)$  a \emph{system}, where $T_1,\ldots, T_\ell\colon$ $ X\to X$ are invertible commuting measure preserving transformations on the probability space $(X,\mathcal{B},\mu).$ 
\end{definition*}

The proposition below, the proof of which relies on a deep result due to Green and Tao (\cite{GT}) on the inverse conjecture for the Gowers norm, will provide us with a crucial intermediate step in order to prove Theorems~\ref{T:ssp1} and ~\ref{T:ssp} (we will actually use a very weak version of it for these results).

\begin{proposition}[Proposition~3.2, \cite{K2}]\label{P:1K2}
Let $\ell, m\in \N,$ $(X,\mathcal{B},\mu,T_1,\ldots, T_m)$ be a system, $p_{i,j}\in\R[t]$ be real polynomials, $1\leq i\leq m,$ $1\leq j\leq \ell$ and  $f_1,\ldots,f_\ell\in L^{\infty}(\mu).$

Then,
$$
 \max_{1\leq r\leq W,\;(r,W)=1}\norm{\frac{1}{N}\sum_{n=1}^N(\Lambda_{w,r}'(n)-1)\cdot (\prod_{i=1}^m T_i^{[p_{i,1}(Wn+r)]})f_1\cdot\ldots\cdot(\prod_{i=1}^m T_i^{[p_{i,\ell}(Wn+r)]})f_\ell}_{L^2(\mu)}
$$
converges to $0$ as $N\to\infty$ and then $w\to\infty.$
\end{proposition}
\begin{proof}[Proof of Theorem~\ref{T:ssp1}]
We borrow the arguments from the proof of Theorem~1.3 from \cite{FHK} (see also Theorem~1.3 in \cite{K2}). By Lemma~\ref{L:n}  it suffices to show that the sequence
$$
A(N):=\frac{1}{N}\sum_{n=1}^N \Lambda'(n)\cdot  T^{[q(n)]}f_1\cdot T^{2[q(n)]}f_2\cdot\ldots\cdot  T^{\ell[q(n)]}f_\ell
$$ converges to the same limit as the sequence $\frac{1}{N}\sum_{n=1}^N  T^{n}f_1\cdot T^{2n}f_2\cdot\ldots\cdot  T^{\ell n}f_\ell,
$ where the convergence takes place in $L^2(\mu).$ 
 For $w$ (which gives a corresponding $W$), $r\in\N,$ we define $$B_{w,r}(N):=\frac{1}{N}\sum_{n=1}^N T^{[q(Wn+r)]}f_1\cdot T^{2[q(Wn+r)]}f_1\cdot\ldots\cdot T^{\ell[q(Wn+r)]}f_\ell.$$ For any $\varepsilon>0,$ using Proposition~\ref{P:1K2} for $m=\ell,$ $T_i=T,$ $1\leq i\leq \ell$ and $ p_{i,j}=\left\{ \begin{array}{ll} 0 \quad ,\text{ if } i\leq \ell-j\\ q \quad ,\text{ elsewhere }\end{array} \right. ,$ for sufficiently large $N$ and some $w_0$ we have 
\begin{equation*}\label{E:ff}
\norm{A(W_0 N)-\frac{1}{\phi(W_0)}\sum_{1\leq r\leq W_0,\;(r,W_0)=1}B_{w_0,r}(N)}_{L^2(\mu)}<\varepsilon.
\end{equation*}
Note at this point that for all $W, r\in\mathbb{N}$ we have that 
 $q(Wt+r)\notin c\mathbb{Q}[t]+d$ for $c,d\in \mathbb{R},$ for otherwise $q$ would have the same property contradicting our assumption.

By Theorem~\ref{T:mc}, we have that for any $1\leq r\leq W_0$ the sequence $(B_{w_0,r}(N))_N$ converges to the same limit as the sequence $\frac{1}{N}\sum_{n=1}^N  T^{n}f_1\cdot T^{2n}f_2\cdot\ldots\cdot  T^{\ell n}f_\ell,
$ and since  $$\lim_{N\to\infty}\norm{A(W_0 N+r)-A(W_0 N)}_{L^2(\mu)}=0$$ for every $1\leq r\leq W_0,$  we get the result.
\end{proof}

\begin{proof}[Proof of Theorem~\ref{T:ssp}]
The proof is analogous to the previous one. In this case we define
$
A(N):=\frac{1}{N}\sum_{n=1}^N \Lambda'(n)\cdot  T^{[p_{1}(n)]}f_1\cdot\ldots\cdot  T^{[p_{\ell}(n)]})f_\ell
$ and for $w,r\in\mathbb{N},$ $B_{w,r}(N):=\frac{1}{N}\sum_{n=1}^N T^{[p_{1}(Wn+r)]}f_1\cdot\ldots\cdot T^{[p_{\ell}(Wn+r)]}f_\ell.$ We use Proposition~\ref{P:1K2} for $m=\ell,$ $T_i=T,$ $1\leq i\leq \ell,$ $ p_{i,j}=\left\{ \begin{array}{ll} 0 \quad ,\text{ if } i\neq j\\ p_i \;\;\; ,\text{ if } i=j\end{array} \right. $ and we note that 
 the family $\{\tilde{p}_{1},\ldots, \tilde{p}_{\ell} \},$ where $\tilde{p}_i(t)=p_i(Wt+r),$ is strongly independent for all $W, r\in\mathbb{N}.$

 Indeed, if for some $(\lambda_1,\ldots,\lambda_\ell)\in \mathbb{R}^\ell\setminus\{\vec{0}\},$ $d\in \mathbb{R},$ $q\in \mathbb{Q}[t]$ and $W,r\in\mathbb{N}$ we had $\sum_{i=1}^\ell \lambda_i p_i(Wt+r)=q(t)+d,$ then $\sum_{i=1}^\ell \lambda_i p_i(t)=\tilde{q}(t)+d,$ where $\tilde{q}(t)=q((t-r)/W) \in \mathbb{Q}[t],$ a contradiction to the strong independence assumption. The result now follows similarly to the previous proof since 
by Theorem~\ref{T:ss}, we have that for any $1\leq r\leq W_0$ the sequence $(B_{w_0,r}(N))_N$ converges, in $L^2(\mu),$ to $\prod_{i=1}^\ell \int f_i\;d\mu.$
\end{proof}



  
  \subsection*{Acknowledgements} We would like to express our indebtedness  to N. Frantzikinakis who introduced us this problem, for his in-depth suggestions and the fruitful discussions we had during the writing procedure of this article. Also, the first author thanks his advisor V. Farmaki for her support during the writing of this paper and the second author thanks V. Bergelson for his constant support.

\end{document}